\documentclass[12pt,reqno]{amsart}
\usepackage[margin=0.7in]{geometry}
\usepackage[colorlinks,linkcolor=blue,citecolor=blue,urlcolor=magenta,bookmarks=false,hypertexnames=true]{hyperref}
 \usepackage{amssymb,amsfonts,amscd,amsbsy, color, amsmath}
\usepackage[mathscr]{eucal}
\usepackage{url}
\usepackage{graphicx}
\usepackage{mathtools}

\newtheorem{theorem}{Theorem}[section]
\newtheorem{lemma}[theorem]{Lemma}

\theoremstyle{definition}
\newtheorem{remark}[theorem]{Remark}
\numberwithin{equation}{section}

\makeatletter
\def\imod#1{\allowbreak\mkern5mu({\operator@font mod}\,\,#1)}
\makeatother
\allowdisplaybreaks

\begin{document}
\title[Combinatorial identities]{Combinatorial identities associated with a bivariate generating function for overpartition pairs}
\author{Atul Dixit}

\author{Ankush Goswami}

\address{Discipline of Mathematics, Indian Institute of Technology Gandhinagar, Palaj, Gandhinagar 382055, Gujarat, India}

\email{adixit@iitgn.ac.in}
\email{ankushgoswami3@gmail.com\\agoswami@iitgn.ac.in}
\maketitle
\begin{abstract}
We obtain a three-parameter $q$-series identity that generalizes two results of Chan and Mao. By specializing our identity, we derive new results of combinatorial significance in connection with $N(r, s, m, n)$, a function counting certain overpartition pairs recently introduced by Bringmann, Lovejoy and Osburn. For example, one of our identities gives a closed-form evaluation of a double series in terms of Chebyshev polynomials of the second kind, thereby resulting in an analogue of Euler's pentagonal number theorem. Another of our results expresses a multi-sum involving $N(r, s, m, n)$ in terms of just the partition function $p(n)$. Using a result of Shimura we also relate a certain double series with a weight 7/2 theta series.
\end{abstract}
\section{Introduction and main results}
A \textit{partition} of a natural number $n$ is the number of ways of writing $n$ as a sum of natural numbers in a non-increasing order. The partition function $p(n)$ enumerates the number of partitions of $n$. Euler showed that the generating function of $p(n)$ is
\begin{eqnarray}\label{gfpar}
1+\sum_{n=1}^\infty p(n)q^n=\frac{1}{(q;q)_\infty},
\end{eqnarray}
where we assume throughout that $|q|<1$, and for $A\in\mathbb{C}$, use the notation
\begin{align}\label{poch}
(A)_n=(A;q)_{n}:=\prod_{j=1}^{n}(1-Aq^{j-1}),\hspace{3mm}
(A)_{\infty}=(A;q)_\infty=\lim_{n\to\infty} (A;q)_n.
\end{align}
The function $p(n)$ satisfies amazing congruences discovered by Ramanujan, namely,
\begin{eqnarray}\label{Rama}
p(5n+4)\equiv 0\pmod{5},\hspace{0.5cm}p(7n+5)\equiv 0\pmod{7},\hspace{0.5cm}p(11n+6)\equiv 0\pmod{11}.
\end{eqnarray}
One of the various ways to prove \eqref{Rama} is to note that the generating function on the right-hand side of \eqref{gfpar} is essentially a half-integral weight (meromorphic) modular form. The congruences in \eqref{Rama} then follow easily using the well-known theory of modular forms; see, for example, \cite{GJS}. For more information on this topic, we refer the reader to \cite[Chapter 2]{spirit} and \cite[Chapter 5]{web}.

In \cite{gea266}, Andrews discovered the remarkable smallest parts partition function $\textup{spt}(n)$. It counts the total number of smallest parts in all partitions of $n$. The generating function of spt$(n)$ is given by \cite[p.~138]{gea266}
\begin{eqnarray}\label{gfspt}
\sum_{n=1}^\infty \text{spt}(n)q^n=\sum_{n=1}^\infty\dfrac{q^n}{(1-q^n)^2(q^{n+1};q)_\infty}.
\end{eqnarray}
The generating function on the right-hand side of \eqref{gfspt} is essentially a mock modular form as shown by Folsom and Ono \cite[Lemma 2.1]{FO} but not a modular form. Nevertheless, it is surprising to note that spt$(n)$ satisfies the following remarkable congruences found by Andrews \cite{gea266} which are reminiscent of  Ramanujan's congruences for $p(n)$ given in \eqref{Rama}:
\begin{eqnarray}\label{And}
\textup{spt}(5n+4)\equiv 0\pmod{5},\hspace{0.5cm}\textup{spt}(7n+5)\equiv 0\pmod{7},\hspace{0.5cm}\textup{spt}(13n+6)\equiv 0\pmod{13}.
\end{eqnarray}
To prove \eqref{And}, Andrews first establishes an identity connecting $p(n),\;\text{spt}(n)$ and $N_2(n)$, where $N_2(n)$ is the Atkin-Garvan second rank moment \cite{atkin-garvan}. His
 identity is \cite[Theorem 3]{gea266}
\begin{align}\label{sptpn}
\textup{spt}(n)=np(n)-\frac{1}{2}N_2(n).\end{align}
This identity is, in turn, proven by him by appropriately specializing Watson's $q$-analogue of Whipple's theorem thereby resulting in \cite[p.~138]{gea266}
\begin{equation}\label{watspl}
\sum_{n=0}^{\infty}\frac{(z)_{n}(z^{-1})_{n}q^n}{(q)_{n}}=\frac{(zq)_{\infty}(z^{-1}q)_{\infty}}{(q)_{\infty}^{2}}\left(1+\sum_{n=1}^{\infty}\frac{(-1)^nq^{n(3n+1)/2}(1+q^n)(z)_{n}(z^{-1})_{n}}{(zq)_{n}(z^{-1}q)_{n}}\right).
\end{equation}
The rest of the proof proceeds in a magical fashion and requires the following ``differentiation identity'' of Andrews \cite[Equation (2.4)]{gea266}:
\begin{eqnarray}\label{diff2}
-\dfrac{1}{2}\left[\dfrac{d^2}{dz^2}(zq;q)_\infty(z^{-1}q;q)_\infty\right]_{z=1}=(q;q)_\infty^2\sum_{n=1}^\infty\dfrac{nq^n}{1-q^n},
\end{eqnarray}
with the help of which he obtains
\begin{equation}\label{new}
\sum_{n=1}^{\infty}\frac{q^n}{(1-q^{n})^2(q^{n+1};q)_{\infty}}=\frac{1}{(q;q)_{\infty}}\sum_{n=1}^{\infty}\frac{nq^n}{1-q^n}+\frac{1}{(q;q)_{\infty}}\sum_{n=1}^{\infty}\frac{(-1)^nq^{n(3n+1)/2}(1+q^n)}{(1-q^n)^2},
\end{equation}
which is nothing but the generating function version of \eqref{sptpn}.

Notice that the expression $(zq;q)_\infty(z^{-1}q;q)_\infty$ appearing in \eqref{watspl} is essentially the product involving the variable $z$ occurring in the Jacobi triple product identity. Our present work was realized from our quest to seek analogues of \eqref{watspl} and \eqref{diff2} together wherein the expression $(zq;q)_\infty(z^{-1}q;q)_\infty$ is replaced by the analogous expression
\begin{align}\label{Dzq}
D(z,q):=(zq;q)_\infty(z^{-1}q;q)_\infty(z^2q;q^2)_\infty(z^{-2}q;q^2)_\infty=\frac{(z^2q;q)_\infty(z^{-2}q;q)_\infty}{(-zq;q)_\infty(-z^{-1}q;q)_\infty},
\end{align}
which arises in the quintuple product identity (see \eqref{QPI1} below). One of the reasons this is important is because \eqref{watspl} and \eqref{diff2} were instrumental in obtaining \eqref{sptpn}.  The desired analogue of \eqref{watspl} is stated in the theorem below, that is,
\begin{theorem}\label{spt-ana}
For $z\in\mathbb{C}$ and $z\not\in\{0,e^{\pm\pi i/3}, -q^{j}, j\in\mathbb{Z}\setminus\{0\}\}$ and $|q|<1$,
\begin{align}\label{spt-anaeqn}
\scalebox{0.95}{$\displaystyle\sum_{n\geq 0}\dfrac{(z^{2};q)_n(z^{-2};q)_nq^n}{(-zq;q)_n(-z^{-1}q;q)_n}=(1-z)(1-z^{-1})\left[\dfrac{-1}{z(1-z^{-1}+z^{-2})}\cdot\dfrac{(z^{-2}q,z^2q)_\infty}{(-z^{-1}q,-zq)_\infty}\right]+\dfrac{(1+z^{-1})}{z(1+z^{-3})}.$}
\end{align}
\end{theorem}
The corresponding analogue of \eqref{diff2} is
\begin{theorem}\label{Quintupledoublederi}
We have
\begin{align*}
&\left[\dfrac{d^2}{dz^2}(zq;q)_\infty(z^{-1}q;q)_\infty(z^2q;q^2)_\infty(z^{-2}q;q^2)_\infty\right]_{z=1}\nonumber\\
&=-2(q;q)_\infty^2(q;q^2)_\infty^2\left\{3\sum_{n=1}^\infty\dfrac{nq^n}{1-q^n}+2\sum_{n=1}^\infty\dfrac{(2n-1)q^{2n-1}}{1-q^{2n-1}}\right\}.
\end{align*}
\end{theorem}
Observe that the expression $D(z,q)$ in \eqref{Dzq} occurs in Theorems \ref{spt-ana} and \ref{Quintupledoublederi} in a way similar to how $(zq;q)_\infty(z^{-1}q;q)_\infty$ appears in \eqref{watspl} and \eqref{diff2}. Differentiating both sides of \eqref{spt-anaeqn} with respect to $z$ twice and letting $z=1$ leads to an analogue of \eqref{new} given below:
\begin{eqnarray}\label{blospt}
4\sum_{n=1}^\infty\dfrac{(q)_{n-1}^2q^n}{(-q)_n^2}&=&-\dfrac{(q)_\infty^2}{(-q)_\infty^2}+1=-\dfrac{\eta(\tau)^4}{\eta(2\tau)^2}+1.
\end{eqnarray}
It is easy to see that the expression on the left-hand side is (essentially) a modular form. The identity \eqref{blospt} is already obtained in \cite[Remark 1.4]{BLO}. However, one of our goals in this paper was to obtain the intermediate identity \eqref{spt-anaeqn} which is not given in \cite{BLO}, for, it gives, as special cases, some new results of combinatorial significance.

Before we discuss these new results, we show that Theorem \ref{spt-ana} follows as a special case of a more general identity which we establish in the following theorem.
\begin{theorem}\label{spt-general}
For $\alpha, \gamma\in\mathbb{C}$, and $\beta\in\mathbb{C}$ except possibly in the set $\{0, \alpha q, \gamma q, q^{-j}, \alpha\gamma q^{j+2}: j\geq 0\}$, we have

\begin{eqnarray}\label{mainid}
\sum_{n=0}^\infty\dfrac{(\alpha,\gamma)_n}{(\beta,\alpha\gamma q^2/\beta)_n}q^n=\dfrac{\beta^{-1}q}{(1-\alpha q/\beta)(1-\gamma q/\beta)}\cdot \dfrac{(\alpha,\gamma)_\infty}{(\beta,\alpha\gamma q^2/\beta)_\infty}+\dfrac{(1-q/\beta)(1-\alpha\gamma q/\beta)}{(1-\gamma q/\beta)(1-\alpha q/\beta)}.
\end{eqnarray}
\end{theorem}
Chan and Mao \cite[Theorem 1.2]{chanmao} recently established the following two $q$-series identities.
\begin{align}
\sum_{n=0}^{\infty}\frac{(x,1/x;q)_nq^n}{(zq,q/z;q)_n}&=\frac{(1-z)^2}{(1-z/x)(1-xz)}+\frac{z(x,1/x;q)_{\infty}}{(1-z/x)(1-xz)(zq,q/z;q)_{\infty}},\label{chanmao1}\\
\sum_{n=0}^{\infty}\frac{(x,q/x;q)_nq^n}{(z,q/z;q)_{n+1}}&=\frac{1}{x(1-z/x)(1-q/(xz))}+\frac{(x,q/x;q)_{\infty}}{z(1-x/z)(1-q/(xz))(z,q/z;q)_{\infty}}.\label{chanmao2} 
\end{align}
We obtain these identities as special cases of Theorem \ref{spt-general}.

Theorem \ref{spt-ana}, in turn, gives the closed-form evaluations of certain bibasic sums such as
\begin{eqnarray}\label{spt-ana11}
4\sum_{n=0}^\infty\dfrac{(-q)_{n-1}^2q^n}{(-q^2;q^2)_n}=2\dfrac{(-q)_\infty^2}{(-q^2;q^2)_\infty}-1=2\dfrac{(q^2;q^2)_\infty^3}{(q)_\infty^2(q^4;q^4)_\infty}-1=2\dfrac{(q^2;q^4)_\infty}{(q;q^2)_\infty^2}-1,
\end{eqnarray}
and 
\begin{eqnarray}\label{spt-ana12}
1+3\sum_{n=1}^\infty\dfrac{(-q)_n(q^3;q^3)_{n-1}q^n}{(q)_{n-1}(-q^3;q^3)_n}=\dfrac{3}{2}\cdot\dfrac{(q^3;q^3)_\infty(-q)_\infty}{(-q^3;q^3)_\infty(q)_\infty}-\dfrac{1}{2}.
\end{eqnarray}
We note that \eqref{spt-ana11} is obtainable from \eqref{chanmao1} by letting $x=-1$ and $z=i$. Also, \eqref{spt-ana12} follows by letting $z=\pm \omega=\pm e^{2\pi i/3}$ in Theorem \ref{spt-ana}.

We now return to Theorem \ref{spt-ana}. The coefficients of the bivariate generating series in this theorem are connected to certain overpartition pairs considered by Bringmann, Lovejoy and Osburn \cite{BLO}. This is now explained.

An overpartition $\lambda$ of a positive integer $n$ is a partition in which the first (or the last) occurrence of a number may be overlined. An overpartition pair $(\lambda,\mu)$ of $n$ is a pair of overpartitions where the sum of all of the parts of $\lambda$ as well as $\mu$ is $n$. Let $\ell((\lambda, \mu))$ denote the largest part of the overpartition pair $(\lambda, \mu)$, that is, the maximum of the largest parts of $\lambda$ and $\mu$. Also, let $n(\pi)$ denote the number of parts of the partition $\pi$. Then the rank of an overpartition pair $(\lambda,\mu)$ is defined by
\begin{equation*}
\ell((\lambda,\mu))-n(\lambda)-n(\mu)-\chi((\lambda,\mu))    
\end{equation*}
where $\chi((\lambda,\mu))$ is defined to be 1 if the largest part of $(\lambda,\mu)$ is non-overlined and is in $\mu$, and
0 otherwise.

Let $N(r,s,m,n)$ denote the number of overpartition pairs of $n$ having rank $m$ such that $r$ is the number of overlined parts in $\lambda$ plus the number of non-overlined parts
in $\mu$, and $s$ is the number of parts in $\mu$. By specializing a result in \cite{Lovejoy}, it was shown in \cite{BLO} that
\begin{equation}\label{BLOrank}
N(d,e,z;q):=\sum_{\substack{r,s,n\geq 0\\m\in\mathbb{Z}}}N(r,s,m,n)d^re^sz^mq^n=\sum_{n\geq 0}\dfrac{(-1/d,-1/e)_n(deq)^n}{(zq,q/z)_n}. 
\end{equation}
This leads to 
\begin{eqnarray}\label{corBLO}
\sum_{n\geq 0}\dfrac{(z^{2};q)_n(z^{-2};q)_nq^n}{(-zq;q)_n(-z^{-1}q;q)_n}
&=&\sum_{\substack{r,s,n\geq 0\\m\in\mathbb{Z}}}N(r,s,m-2s+2r,n)(-1)^{r+s+m}z^m q^n.
\end{eqnarray}
More generally, letting $d=-x=e^{-1}$ in \eqref{BLOrank}, one can represent the left-hand side of \eqref{chanmao1} in terms of the function $N(r, s, m, n)$.

Using Theorem \ref{spt-ana} and \eqref{corBLO}, we obtain a closed-form evaluation of a double series involving $N(r,s,m,n)$ in terms of Chebyshev polynomials of the second kind $U_n(x)$ (see Section \ref{NP} for the definition and properties of $U_n(x)$). Before stating this result, we discuss the necessary setup. We define the following subsets of integers:

\begin{align}\label{subint}
&I_1:=(-\infty,-3n)\cap\mathbb{Z},\hspace{0.5cm}I_2:=\{-3n\},\hspace{0.5cm}I_3:=(-3n,1)\cap\mathbb{Z},\notag\\&I_4:=[1,3n+1)\cap\mathbb{Z},\hspace{0.5cm}I_5:=\{3n+1\},\hspace{0.5cm}I_6:=(3n+1,\infty)\cap\mathbb{Z},
\end{align}
and 
\begin{align}\label{subint1}
&I_1':=(-\infty,-3n]\cap\mathbb{Z},\hspace{0.5cm}I_2':=\{-3n+1\},\hspace{0.5cm}I_3':=(-3n+1,1)\cap\mathbb{Z},\notag\\&I_4':=[1,3n)\cap\mathbb{Z},\hspace{0.5cm}I_5':=\{3n\},\hspace{0.5cm}I_6':=[3n+1,\infty)\cap\mathbb{Z}.
\end{align}
Then we have the following result.
\begin{theorem}\label{BLOcoeff}
For $\ell\in\mathbb{Z}$, let $\omega_\ell:=\frac{3\ell^2+\ell}{2}$ be a pentagonal number. Then the coefficient of $z^mq^n$ in $(q)_{\infty}\displaystyle\sum_{n\geq 0}\dfrac{(z^{2};q)_n(z^{-2};q)_nq^n}{(-zq;q)_n(-z^{-1}q;q)_n}$ is $\displaystyle\sum_{\substack{r,s\geq 0\\0\leq \omega_k\leq n}}N\left(r,s,m-2s+2r,n-\omega_k\right)(-1)^{r+s+m+k}$,
and is given by
\begin{equation*}
\begin{cases}
0,&m\in I_1\\
U_1(1/2),&m\in I_2\\
U_{m+3n+1}(1/2),&m\in I_3\\
U_{m+3n+1}(1/2)+(-1)^nU_{m-1}(1/2),&m\in I_4\\
-U_1(1/2)+U_{6n+2}(1/2)+(-1)^nU_{3n}(1/2),&m\in I_5\\
-U_{m-3n}(1/2)+U_{m+3n+1}(1/2)+(-1)^nU_{m-1}(1/2),&m\in I_6
\end{cases}   
\end{equation*}
when $n=\omega_\ell, \;(\ell\geq 0)$, and by
\begin{equation*}
\begin{cases}
0,&m\in I_1'\\
-U_1(1/2),&m\in I_2'\\
-U_{m+3n}(1/2),&m\in I_3'\\
-U_{m+3n}(1/2)+(-1)^nU_{m-1}(1/2),&m\in I_4'\\
-U_{6n}(1/2)+U_1(1/2)+(-1)^nU_{3n-1}(1/2),&m\in I_5'\\
-U_{m+3n}(1/2)+U_{m-3n+1}(1/2)+(-1)^nU_{m-1}(1/2),&m\in I_6'
\end{cases}
\end{equation*}
when $n=\omega_{-\ell}, \;(\ell\geq 1)$.
\end{theorem}
We also obtain a combinatorial identity expressing a multi-sum involving $N(r,s,m,n)$ explicitly in terms of just the partition function $p(n)$.
\begin{theorem}\label{ideNp}
Let 
\begin{equation}\label{an}
a(n):=\sum_{k=0}^{2n}(-1)^kp(k)p(2n-k).  
\end{equation}
Then
\begin{eqnarray*}
&&-\frac{1}{24}\sum_{\substack{r,s\geq 0\\m\in\mathbb{Z}}}\sum_{1\leq j\leq n}(-1)^{r+s+m}m^2(m^2-11)a(j)N(r,s,m-2s+2r,n-j).\notag\\
&&\hspace{3cm}=\sum_{k=-\infty} ^{\infty}(-1)^k\left(3(n-k^2)p(n-k^2)-2n(-1)^np(n-2k^2)\right).
\end{eqnarray*}
\end{theorem}
\begin{remark}
Observe that both sides of the above identity are finite sums. 
\end{remark}
Next, using Theorem \ref{spt-ana} and Theorem \ref{Quintupledoublederi}, we express a double series in terms of a linear combination of single series as follows:
\begin{theorem}\label{Th3}
We have
\begin{eqnarray*}
&&5\sum_{n=1}^\infty\dfrac{(q)_{n-1}^2q^n}{(-q)_n^2}-4\sum_{n=1}^\infty\dfrac{(q)_{n-1}^2q^n}{(-q)_n^2}\left(\sum_{k=1}^{n-1}\dfrac{q^k(5+6q^k+5q^{2k})}{(1-q^{2k})^2}+\dfrac{q^n}{(1+q^n)^2}\right)\notag\\
&&\hspace{2cm}=\dfrac{(q)_\infty^2}{(-q)_\infty^2}\left\{3\sum_{n=1}^\infty\dfrac{nq^n}{1-q^n}+2\sum_{n=1}^\infty\dfrac{(2n-1)q^{2n-1}}{1-q^{2n-1}}\right\}.
\end{eqnarray*}
\end{theorem}
Using a result of Shimura, the double series in Theorem \ref{Th3} can be expressed in terms of a linear combination of modular forms involving a unary theta series of weight $7/2$. 
\begin{theorem}\label{spt-mod}
We have
\begin{eqnarray*}
4\sum_{n=1}^\infty\dfrac{(q)_{n-1}^2q^n}{(-q)_n^2}\left(\sum_{k=1}^{n-1}\dfrac{q^k(5+6q^k+5q^{2k})}{(1-q^{2k})^2}+\dfrac{q^n}{(1+q^n)^2}\right)
=\dfrac{5}{4}-\dfrac{31}{24}\cdot\dfrac{\eta(\tau)^4}{\eta(2\tau)^2}+\dfrac{1}{24}\cdot\dfrac{\theta(\tau)}{\eta(\tau)}
\end{eqnarray*}
where
\begin{equation*}
\theta(\tau):=\sum_{n=1}^\infty\left(\dfrac{n}{12}\right)n^3q^{\frac{n^2}{24}},\hspace{1cm}\left(\dfrac{n}{12}\right):=\begin{cases}
1,&n\equiv 1\pmod{6}\notag\\
-1,&n\equiv 5\pmod{6}\notag\\
0,&\emph{otherwise}
\end{cases}
\end{equation*}
and $\eta(\tau)=q^{\frac{1}{24}}(q;q)_\infty$ is the Dedekind eta-function. Also, $\theta(2\tau)$ satisfies
\begin{equation}\label{thetatrans}
\theta\left(2\cdot\dfrac{a\tau+b}{c\tau+d}\right)=e^{\frac{i\pi ab}{6}}\left(\dfrac{3c}{d}\right)\varepsilon_d^{-1}(c\tau+d)^{7/2}\theta(2\tau),\;\begin{pmatrix}
a&b\\c&d
\end{pmatrix}\in\Gamma_1(12).
\end{equation}
Here $\left(\dfrac{\cdot}{\cdot}\right)$ is the extended Jacobi symbol and $\varepsilon_d=1$ or $i$ according as $d\equiv 1$ or $3\;\pmod{4}$ and $\Gamma_1(N)$ is a congruence subgroup of $SL_2(\mathbb{Z})$ consisting of all $2\times 2$ matrices with diagonal entries $\equiv 1\pmod{N}$ and the lower left-entry $\equiv 0\pmod{N}$.
\end{theorem}
\section{Notations and preliminaries}\label{NP}
In addition to \eqref{poch}, we adopt the following notations:
\begin{align}
(\alpha_1,\alpha_2,\alpha_3,\cdots,\alpha_k;q)_n=\prod_{i=1}^k(\alpha_i;q)_n,\hspace{0.5cm}
(\alpha_1,\alpha_2,\alpha_3,\cdots,\alpha_k;q)_\infty=\lim_{n\rightarrow\infty}(\alpha_1,\alpha_2,\alpha_3,\cdots,\alpha_k;q)_n.
\end{align}
We also require the following \textit{unilateral basic hypergeometric series}:
\begin{eqnarray}
_{{k+1}}\phi _{k}\left({\begin{matrix}a_{1}&a_{2}&\ldots &a_{{k}}&a_{{k+1}}\\\mbox{}&b_{1}&b_{2}&\ldots &b_{{k}}\end{matrix}};q,z\right)=\sum _{{n=0}}^{\infty }{\frac  {(a_{1},a_{2},\ldots ,a_{{k+1}};q)_{n}}{(b_{1},b_{2},\ldots ,b_{k},q;q)_{n}}}z^{n}.
\end{eqnarray}
The Chebyshev polynomials of the second kind $U_n(x)$ are defined by the recurrence relation \cite[p.~9, Equation (1.2.15 (a)-(b))]{Rivlin}
\begin{eqnarray}\label{Cherec}
U_0(x)=1,\hspace{1cm}U_1(x)=2x,\hspace{1cm}U_{n+1}(x)=2x\;U_n(x)-U_{n-1}(x).
\end{eqnarray}
The ordinary generating function of $U_n(x)$ is \cite[p.~155, Equation (6.45) with $\gamma=1$]{Temme}
\begin{eqnarray}\label{Chesec}
\sum_{n=0}^\infty U_n(x)t^n=\dfrac{1}{1-2tx+t^2}.
\end{eqnarray}
Also we have \cite[p.~7, Equation (1.23)]{Rivlin}
\begin{eqnarray}\label{Unexprr}
U_n(x)=\dfrac{1}{n+1}T'_{n+1}(x)=\dfrac{\sin((n+1)\cos^{-1}(x))}{\sin(\cos^{-1}(x))}
\end{eqnarray}
where $T_k(x)$ is a Chebyshev polynomial of the first kind with the generating function \cite[p.~36, Equation (1.105)]{Rivlin}
\begin{equation*}
\sum_{n=0}^\infty T_n(x)t^n=\dfrac{1-tx}{1-2tx+t^2}.    
\end{equation*}
From \eqref{Chesec}, it follows that for $n>1$ 
\begin{eqnarray}\label{Chebynega}
U_{-n}(x)=-U_{n-2}(x),\hspace{2cm}U_{-1}(x)=0.
\end{eqnarray}
Next, we require a result due to Agarwal \cite[Equation (3.1)]{Aga}. 
\begin{theorem}\label{Aga-id}
We have
\begin{eqnarray}
\sum_{n=0}^\infty\dfrac{(\alpha)_n(\gamma)_n}{(\beta)_n(\delta)_n}t^n&=&\dfrac{(q/(\alpha t),\gamma,\alpha t,\beta/\alpha,q;q)_\infty}{(\beta/(\alpha t), \delta, t, q/\alpha,\beta;q)_\infty}\;{}_2\phi_1\left(\begin{matrix}
\delta/\gamma,&t\\
&q\alpha t/\beta
\end{matrix}
;q,\;\gamma q/\beta\right)\notag\\&+&\;\dfrac{(\gamma)_\infty}{(\delta)_\infty}\left(1-\dfrac{q}{\beta}\right)\sum_{m=0}^\infty\dfrac{(\delta/\gamma)_m(t)_m}{(q)_m(\alpha t/\beta)_{m+1}}(q\gamma/\beta)^m\left({}_2\phi_1\left(\begin{matrix}
q,&q/t\\
&q\beta/(\alpha t)
\end{matrix}
;q,\;q/\alpha\right)-1\right)\notag\\
&+&\;\dfrac{(\gamma)_\infty}{(\delta)_\infty}\left(1-\dfrac{q}{\beta}\right)\sum_{p=0}^\infty\dfrac{\gamma^p(\delta/\gamma)_p}{(q)_p}\sum_{m=0}^\infty\dfrac{(\delta q^p/\gamma)_m(tq^p)_m}{(q^{1+p})_m(\alpha tq^p/\beta)_{m+1}}(q\gamma/\beta)^m.
\end{eqnarray}
\end{theorem}
The next result generalizes an identity of Andrews \cite{gea266} in the case $k=2$. It is required later in our proofs. 
\begin{lemma}\label{diffk}
For any $C^{\infty}$-function $f$ and $2\leq k\in\mathbb{N}$, we have
\begin{eqnarray}
-\dfrac{(-1)^k}{k!}\Bigg[\dfrac{d^k}{dz^k}(1-z)(1-z^{-1})f(z)\Bigg]_{z=1}=\sum_{\ell=2}^k\dfrac{(-1)^\ell}{(\ell-2)!}\Bigg[\dfrac{d^{\ell-2}}{dz^{\ell-2}}f(z)\Bigg]_{z=1}.
\end{eqnarray}
\end{lemma}
\begin{proof}
The proof follows by Taylor's expansion and successive differentiation.
\end{proof}
\begin{lemma}\label{parlemma}
We have
\begin{eqnarray*}
\dfrac{(-q;q)_\infty}{(q;q)_\infty^2}=:\sum_{n=0}^\infty a(n)q^n
\end{eqnarray*}
where $a(n)$ is defined in \eqref{an}.
\end{lemma}
\begin{proof}
Observe that
\begin{eqnarray}\label{eq1}
\dfrac{1}{(q;q)_\infty}\prod_{n=1}^\infty\dfrac{1}{1-(-q)^n}=\dfrac{1}{(q;q)_\infty(q^2;q^2)_\infty(-q;q^2)_\infty}=\dfrac{(q^4;q^4)_\infty}{(q^2;q^2)_\infty^3}=:\sum_{n\geq 0}a(n)q^{2n},
\end{eqnarray}
say. On the other hand,
\begin{align}\label{eq2}
\dfrac{1}{(q;q)_\infty}\prod_{n=1}^\infty\dfrac{1}{1-(-q)^n}=\left(\sum_{\ell=0}^\infty p(\ell)q^\ell\right)\left(\sum_{k=0}^\infty (-1)^k p(k)q^k\right)=\sum_{n\geq 0}\left(\sum_{k=0}^n(-1)^kp(k)p(n-k)\right)q^n.
\end{align}
Comparing coefficients of $q^{2n}$ on both sides of \eqref{eq1} and \eqref{eq2} yields the result. Additionally, we obtain
\begin{eqnarray}
\sum_{k=0}^{2n+1}(-1)^kp(k)p(2n+1-k)=0,
\end{eqnarray}
which also follows immediately by rearranging the sum. 
\end{proof}
\section{Proofs of the main results}
\subsection{Proof of Theorem \ref{spt-general}}
We first prove the result for $|\gamma|<\textup{min}\left(1,|\beta/q|\right)$. Let $t=q$ and $\delta=\alpha\gamma q^2/\beta$ in \eqref{Aga-id} to obtain

\begin{align}\label{nid1}
\sum_{n=0}^\infty\dfrac{(\alpha,\gamma)_n}{(\beta,\alpha\gamma q^2/\beta)_n}q^n&=\dfrac{(\alpha^{-1},\gamma,\alpha q,\alpha^{-1}\beta,q)_\infty}{(\alpha^{-1}q^{-1}\beta,\alpha\beta^{-1}\gamma q^2,q,\alpha^{-1}q,\beta)_\infty}\sum_{k=0}^\infty\left(\dfrac{\gamma q}{\beta}\right)^k\notag\\&\quad+\dfrac{(\gamma)_\infty}{(\alpha\beta^{-1}\gamma q^2)_\infty}\left(1-\dfrac{q}{\beta}\right)\sum_{p=0}^\infty\dfrac{\gamma^p(\alpha q^2/\beta)_p}{(q)_p}\sum_{m=0}^\infty\dfrac{(\alpha q^{p+2}/\beta)_m}{(\alpha q^{p+1}/\beta)_{m+1}}\left(\dfrac{\gamma q}{\beta}\right)^m\notag\\
&=\dfrac{(1-\alpha^{-1})}{(1-\alpha^{-1}\beta q^{-1})(1-\beta^{-1}\gamma q)}\cdot \dfrac{(\alpha q,\gamma)_\infty}{(\beta,\alpha\gamma q^2/\beta)_\infty}\notag\\&\quad+\dfrac{(\gamma)_\infty}{(\alpha\beta^{-1}\gamma q^2)_\infty}\left(1-\dfrac{q}{\beta}\right)\sum_{p=0}^\infty\dfrac{\gamma^p(\alpha q^2/\beta)_p}{(1-\alpha q^{p+1}/\beta)(q)_p}\sum_{m=0}^\infty\left(\dfrac{\gamma q}{\beta}\right)^m\notag\\
&=\dfrac{\beta^{-1}q}{(1-\alpha\beta^{-1} q)(1-\beta^{-1}\gamma q)}\cdot \dfrac{(\alpha,\gamma)_\infty}{(\beta,\alpha\gamma q^2/\beta)_\infty}\notag\\
&\quad+\dfrac{(\gamma)_\infty}{(\alpha\beta^{-1}\gamma q^2)_\infty}\cdot\dfrac{(1-\beta^{-1}q)}{(1-\beta^{-1}\gamma q)}\sum_{p=0}^\infty\dfrac{\gamma^p(\alpha q^2/\beta)_p}{(1-\alpha q^{p+1}/\beta)(q)_p},
\end{align}
where in the penultimate step we used the condition $|\gamma q/\beta|<1$.

Next, we notice that the sum in the right-hand side of \eqref{nid1} can be rewritten as

\begin{eqnarray}\label{nid2}
\sum_{p=0}^\infty\dfrac{\gamma^p(\alpha q^2/\beta)_p}{(1-\alpha q^{p+1}/\beta)(q)_p}&=&\dfrac{1}{(1-\alpha\beta^{-1}q)}+\sum_{p=1}^\infty\dfrac{(\alpha q^2/\beta)_{p-1}}{(q)_p}\gamma^p\notag\\
&=&\dfrac{1}{(1-\alpha\beta^{-1}q)}+\dfrac{1}{(1-\alpha\beta^{-1}q)}\sum_{p=1}^\infty\dfrac{(\alpha q/\beta)_{p}}{(q)_p}\gamma^p\notag\\
&=&\dfrac{1}{(1-\alpha\beta^{-1}q)}\sum_{p=0}^\infty\dfrac{(\alpha q/\beta)_{p}}{(q)_p}\gamma^p\notag\\
&=&\dfrac{1}{(1-\alpha\beta^{-1}q)}\cdot\dfrac{(\alpha\gamma q/\beta)_\infty}{(\gamma)_\infty}
\end{eqnarray}
where the last step follows by $q$-binomial theorem \cite[p.~8, Theorem 1.3.1]{spirit} since $|\gamma|<1$. Identity \eqref{mainid} now follows for $|\gamma|<\textup{min}\left(1,|\beta/q|\right)$ from \eqref{nid1} and \eqref{nid2}. By analytic continuation, the result is easily seen to be extended to the said values in the hypotheses.
\qed
\subsection{Proofs of \eqref{chanmao1} and \eqref{chanmao2}}
Equation \eqref{chanmao1} readily follows from Theorem \ref{spt-general} by letting $\alpha=x=\gamma^{-1}$ and $\beta=zq$. Similarly, letting $\alpha=x, \gamma=q/x$ and $\beta=zq$ results in \eqref{chanmao2}.
\qed
\subsection{Proof of Theorem \ref{spt-ana}}

Theorem \ref{spt-ana} follows from \eqref{chanmao1} by first replacing $z$ by $-z$ and then letting $x=z^2$.
\qed
\subsection{Proof of Theorem \ref{Quintupledoublederi}}
The quintuple product identity is given by \cite[p.~18, Theorem 1.3.17]{spirit}
\begin{eqnarray}\label{QPI}
\sum_{n=-\infty}^\infty q^{3n^2+n}\left(z^{3n}q^{-3n}-z^{-3n-1}q^{3n+1}\right)=(q^2;q^2)_\infty(qz;q^2)_\infty(q/z;q^2)_\infty(z^2;q^4)_\infty(q^4/z^2;q^4)_\infty.
\end{eqnarray}
Replacing $q$ by $\sqrt{q}$ and then $z$ by $z\sqrt{q}$ in \eqref{QPI} gives
\begin{eqnarray}\label{QPI1}
\sum_{n=-\infty}^\infty q^{(3n^2+n)/2}\left(z^{3n}-z^{-3n-1}\right)=(q;q)_\infty(zq;q)_\infty(z^{-1};q)_\infty(z^2q;q^2)_\infty(z^{-2}q;q^2)_\infty.
\end{eqnarray}
Hence
\begin{align}\label{2ndderiQPI}
&\left[\dfrac{d^2}{dz^2}(zq;q)_\infty(z^{-1}q;q)_\infty(z^2q;q^2)_\infty(z^{-2}q;q^2)_\infty\right]_{z=1}=\dfrac{1}{(q;q)_\infty}\left[\dfrac{d^2}{dz^2}\dfrac{\sum_{n=-\infty}^\infty q^{(3n^2+n)/2}(z^{3n}-z^{-3n-1})}{1-z^{-1}}\right]_{z=1}\notag\\
&\hspace{5cm}=\dfrac{1}{(q;q)_\infty}\left[\dfrac{d^2}{dz^2}\sum_{n=-\infty}^\infty q^{(3n^2+n)/2}\left(\dfrac{1-z^{6n+1}}{1-z}\right)\right]_{z=1}\notag\\
&\hspace{5cm}=\dfrac{1}{(q;q)_\infty}\left[\dfrac{d^2}{dz^2}\sum_{n=-\infty}^\infty q^{(3n^2+n)/2}\sum_{j=0}^{6n}z^{-3n+j}\right]_{z=1}\notag\\
&\hspace{5cm}=\dfrac{1}{(q;q)_\infty}\left[\dfrac{d^2}{dz^2}\sum_{n=-\infty}^\infty q^{(3n^2+n)/2}\sum_{j=0}^{6n}(-3n+j)(-3n+j-1)z^{-3n+j-2}\right]_{z=1}\notag\\
&\hspace{5cm}=\dfrac{1}{(q;q)_\infty}\sum_{n=-\infty}^\infty q^{(3n^2+n)/2}n(3n+1)(6n+1)\notag\\
&\hspace{5cm}=\dfrac{2q}{(q;q)_\infty}\sum_{n=-\infty}^\infty (6n+1)\frac{n(3n+1)}{2}q^{(3n^2+n)/2-1}\notag\\
&\hspace{5cm}=\dfrac{2q}{(q;q)_\infty}\dfrac{d}{dq}\sum_{n=-\infty}^\infty (6n+1)q^{(3n^2+n)/2}.
\end{align}
Now employing an identity of Ramanujan and proved by Gordon \cite[p. 20, Corollary 1.3.21]{spirit},
\begin{eqnarray}\label{ano-id}
\sum_{n=-\infty}^\infty (6n+1)q^{(3n^2+n)/2}=(q;q)_\infty^3(q;q^2)_\infty^2.
\end{eqnarray}
Thus, \eqref{2ndderiQPI} and \eqref{ano-id} give
\begin{eqnarray*}
&&\left[\dfrac{d^2}{dz^2}(zq;q)_\infty(z^{-1}q;q)_\infty(z^2q;q^2)_\infty(z^{-2}q;q^2)_\infty\right]_{z=1}=\dfrac{2q}{(q;q)_\infty}\dfrac{d}{dq}(q;q)_\infty^3(q;q^2)_\infty^2\notag\\
&&\hspace{4cm}=-2(q;q)_\infty^2(q;q^2)_\infty^2\left\{3\sum_{n=1}^\infty\dfrac{nq^n}{1-q^n}+2\sum_{n=1}^\infty\dfrac{(2n-1)q^{2n-1}}{1-q^{2n-1}}\right\}.
\end{eqnarray*}
\qed

\subsection{Proof of Theorem \ref{BLOcoeff}}
Using the Quintuple product identity \eqref{QPI1} on the right-hand side of Theorem \ref{spt-ana}, we obtain
\begin{eqnarray}\label{rhs-spt-ana}
&&(1-z)(1-z^{-1})\left[\dfrac{-1}{z(1-z^{-1}+z^{-2})}\cdot\dfrac{(z^{-2}q,z^2q)_\infty}{(-z^{-1}q,-zq)_\infty}\right]+\dfrac{(1+z^{-1})}{z(1+z^{-3})}\notag\\
&=& -\dfrac{(1-z)(1-z^{-1})}{z(1-z^{-1}+z^{-2})}\cdot\dfrac{\displaystyle\sum_{n=-\infty}^\infty q^{(3n^2+n)/2}(z^{3n}-z^{-3n-1})}{(1-z^{-1})(q)_\infty}+\dfrac{1}{z(1-z^{-1}+z^{-2})}\notag\\
&=&-\dfrac{1}{z(1-z^{-1}+z^{-2})}\left[\dfrac{(1-z)}{(q)_\infty}\sum_{n=-\infty}^\infty q^{(3n^2+n)/2}(z^{3n}-z^{-3n-1})-1\right].
\end{eqnarray}
Substituting $x=1/2$ and $t=z$ in \eqref{Chesec}, we obtain
\begin{equation}\label{Chesec1}
\sum_{n=0}^\infty U_n\left(\frac{1}{2}\right)z^n=\dfrac{1}{1-z+z^2}.    
\end{equation}
Thus, \eqref{Cherec}, \eqref{rhs-spt-ana} and \eqref{Chesec1} yield
\begin{eqnarray}\label{rhs-spt-Che}
&&(1-z)(1-z^{-1})\left[\dfrac{-1}{z(1-z^{-1}+z^{-2})}\cdot\dfrac{(z^{-2}q,z^2q)_\infty}{(-z^{-1}q,-zq)_\infty}\right]+\dfrac{(1+z^{-1})}{z(1+z^{-3})}\notag\\&=&\left(-z\sum_{m=0}^\infty U_m\left(\dfrac{1}{2}\right)z^m\right)\left[\dfrac{(1-z)}{(q)_\infty}\sum_{n=-\infty}^\infty q^{(3n^2+n)/2}(z^{3n}-z^{-3n-1})-1\right]\notag\\
&=&-\dfrac{(1-z)}{(q)_\infty}\sum_{m=1}^\infty U_{m-1}(1/2)z^m\sum_{n=-\infty}^\infty q^{(3n^2+n)/2}(z^{3n}-z^{-3n-1})+\sum_{m=1}^\infty U_{m-1}(1/2)z^m\notag\\
&=&-\dfrac{1}{(q)_\infty}\left(\sum_{m=1}^\infty U_{m}(1/2)z^m\right)\sum_{n=-\infty}^\infty q^{(3n^2+n)/2}(z^{3n}-z^{-3n-1})+\sum_{m=1}^\infty U_{m-1}(1/2)z^m\notag\\
&=&-\dfrac{1}{(q)_\infty}\sum_{m=1}^\infty\sum_{n=-\infty}^\infty U_{m}(1/2)q^{(3n^2+n)/2}(z^{m+3n}-z^{m-3n-1})+\sum_{m=1}^\infty U_{m-1}(1/2)z^m.\notag\\
&=&-\dfrac{1}{(q)_\infty}\left(\sum_{n=-\infty}^\infty \sum_{m=1}^\infty U_{m}(1/2)q^{(3n^2+n)/2}z^{m+3n}-\sum_{n=-\infty}^\infty \sum_{m=1}^\infty U_{m}(1/2)q^{(3n^2+n)/2}z^{m-3n-1}\right)\notag\\
&&+\sum_{m=1}^\infty U_{m-1}(1/2)z^m.
\end{eqnarray}
Making the change of variables $m\rightarrow m-3n$ and $m\rightarrow m+3n+1$ respectively on the two double series in the right-hand side of \eqref{rhs-spt-Che}, we obtain
\begin{eqnarray}\label{simplifyrhs}
&&(1-z)(1-z^{-1})\left[\dfrac{-1}{z(1-z^{-1}+z^{-2})}\cdot\dfrac{(z^{-2}q,z^2q)_\infty}{(-z^{-1}q,-zq)_\infty}\right]+\dfrac{(1+z^{-1})}{z(1+z^{-3})}\notag\\&&=-\dfrac{1}{(q)_\infty}\left(\sum_{n=-\infty}^\infty \sum_{m=3n+1}^\infty U_{m-3n}(1/2)q^{(3n^2+n)/2}z^{m}-\sum_{n=-\infty}^\infty \sum_{m=-3n}^\infty U_{m+3n+1}(1/2)q^{(3n^2+n)/2}z^{m}\right)\notag\\
&&+\sum_{m=1}^\infty U_{m-1}(1/2)z^m.
\end{eqnarray}
Using Theorem \ref{spt-ana} and \eqref{corBLO}, we immediately see that the left-hand side of \eqref{simplifyrhs} can be rewritten in the following way and we have:
\begin{eqnarray}\label{BLO-lhs}
&&\sum_{\substack{r,s,n\geq 0\\m\in\mathbb{Z}}}N(r,s,m-2s+2r,n)(-1)^{r+s+m}z^m q^n\notag\\
&&=-\dfrac{1}{(q)_\infty}\left(\sum_{n=-\infty}^\infty \sum_{m=3n+1}^\infty U_{m-3n}(1/2)q^{(3n^2+n)/2}z^{m}-\sum_{n=-\infty}^\infty \sum_{m=-3n}^\infty U_{m+3n+1}(1/2)q^{(3n^2+n)/2}z^{m}\right)\notag\\
&&+\sum_{m=1}^\infty U_{m-1}(1/2)z^m.
\end{eqnarray}
Multiplying both sides of \eqref{BLO-lhs} by $(q)_\infty$, employing Euler's pentagonal number theorem $(q)_{\infty}=\sum_{k=-\infty}^{\infty}(-1)^kq^{k(3k+1)/2}$, and simplifying, we get
\begin{eqnarray}\label{f15}
&&\sum_{\substack{n\geq 0\\m\in\mathbb{Z}}}\left(\sum_{\substack{r,s\geq 0\\0\leq \omega_k\leq n}}N\left(r,s,m-2s+2r,n-\omega_k\right)(-1)^{r+s+m+k}\right)z^m q^n\notag\\
&&=-\sum_{n=-\infty}^\infty \sum_{m=3n+1}^\infty U_{m-3n}(1/2)q^{(3n^2+n)/2}z^{m}+\sum_{n=-\infty}^\infty \sum_{m=-3n}^\infty U_{m+3n+1}(1/2)q^{(3n^2+n)/2}z^{m}\notag\\
&&+\sum_{n=-\infty}^\infty\sum_{m=1}^\infty(-1)^n U_{m-1}(1/2)z^mq^{(3n^2+n)/2}.
\end{eqnarray}
We define the following functions:
\begin{eqnarray}\label{npos}
&&f_1(m,n):=\begin{cases}
U_{m-3n}(1/2),&m>3n\\
0,&\text{otherwise},
\end{cases}
\hspace{0.7cm}
f_2(m,n):=\begin{cases}
U_{m+3n+1}(1/2),&m\geq -3n\\
0,&\text{otherwise},
\end{cases}\notag\\
&&f_3(m,n):=\begin{cases}
(-1)^nU_{m-1}(1/2),&m\geq 1\\
0,&\text{otherwise}
\end{cases}
\end{eqnarray}
\begin{eqnarray}\label{nneg}
&&\tilde{f}_1(m,n):=\begin{cases}
U_{m+3n}(1/2),&m>-3n\\
0,&\text{otherwise},
\end{cases}
\hspace{0.7cm}
\tilde{f}_2(m,n):=\begin{cases}
U_{m-3n+1}(1/2),&m\geq 3n\\
0,&\text{otherwise},
\end{cases}\notag\\
&&\tilde{f}_3(m,n):=\begin{cases}
(-1)^nU_{m-1}(1/2),&m\geq 1\\
0,&\text{otherwise}.
\end{cases}
\end{eqnarray}
Consider the first double series
\begin{align}
&-\sum_{n=-\infty}^\infty \sum_{m=3n+1}^\infty U_{m-3n}(1/2)q^{(3n^2+n)/2}z^{m}   \nonumber\\ &=-\sum_{n=1}^{\infty}\sum_{m=-3n+1}^\infty U_{m+3n}(1/2)q^{(3n^2-n)/2}z^{m}-\sum_{n=0}^{\infty} \sum_{m=3n+1}^\infty U_{m-3n}(1/2)q^{(3n^2+n)/2}z^{m}\nonumber\\
&=-\sum_{n=1}^{\infty}\sum_{m=-\infty}^{\infty}\tilde{f_{1}}(m, n)z^mq^{\omega_{-n}}-\sum_{n=0}^{\infty}\sum_{m=-\infty}^{\infty}f_{1}(m, n)z^mq^{\omega_{n}}.
\end{align}
Similarly the remaining two double series in \eqref{f15} can be written using $f_{2}(m, n)$, $\tilde{f_{2}}(m, n)$, $f_{3}(m, n)$ and $\tilde{f_{3}}(m, n)$ thereby leading to
\begin{eqnarray}\label{f16}
&&\sum_{\substack{n\geq 0\\m\in\mathbb{Z}}}\left(\sum_{\substack{r,s\geq 0\\0\leq \omega_k\leq n}}N\left(r,s,m-2s+2r,n-\omega_k\right)(-1)^{r+s+m+k}\right)z^m q^n\notag\\
&=&\sum_{n=0}^{\infty}\sum_{m=-\infty}^{\infty}\left(-f_1(m, n)+f_2(m, n)+f_3(m, n)\right)z^{m}q^{\omega_n}\notag\\
&&+\sum_{n=1}^{\infty}\sum_{m=-\infty}^{\infty}\left(-\tilde{f_1}(m, n)+\tilde{f_2}(m, n)+\tilde{f_3}(m, n)\right)z^{m}q^{\omega_{-n}}.
\end{eqnarray}
This establishes the result. \qed

Before proving Theorem \ref{ideNp}, we establish a crucial lemma which also appears to be new. One of the ideas employed in its proof resulted through a personal communication with George Andrews \cite{andrewsmail}.
\begin{lemma}\label{claim1lemma}
Let $p_{\text{sc}}(n)$ denote the number of self-conjugate partitions of $n$. Then
\begin{eqnarray}\label{claim1}
p_{sc}(n)=p(n)+2\sum_{j\geq 1}(-1)^jp(n-2j^2).
\end{eqnarray}
\end{lemma}
\begin{proof}
To prove \eqref{claim1}, we need the two identities below
\begin{equation}\label{disspar}
\sum_{n=0}^\infty p(2n)q^n=\dfrac{(-q^3,-q^5,q^8;q^8)_\infty}{(q)_\infty^2},\hspace{1cm}\sum_{n=0}^\infty p(2n+1)q^n=\dfrac{(-q,-q^7,q^8;q^8)_\infty}{(q)_\infty^2},   
\end{equation}
which follow from the $2$-dissection of Gauss' triangular series identity:
\begin{eqnarray}
\psi(q):=\sum_{n=-\infty}^\infty q^{2n^2-n}=\dfrac{(q^2;q^2)_\infty^2}{(q;q)_\infty}. 
\end{eqnarray}
At this point, we note that \cite[p.~5430]{andrewsballantine}
\begin{eqnarray}
\sum_{n\geq 0}p_{\text{sc}}(n)q^n=\sum_{n\geq 0}\dfrac{q^{n^2}}{(q^2;q^2)_n},
\end{eqnarray}
which yields the following identities upon $2$-dissections:
\begin{eqnarray}\label{scdiss}
\sum_{n\geq 0}p_{\text{sc}}(2n)q^n=\sum_{n\geq 0}\dfrac{q^{2n^2}}{(q;q)_{2n}},\hspace{1cm}\sum_{n\geq 0}p_{\text{sc}}(2n+1)q^n=\sum_{n\geq 0}\dfrac{q^{2n^2+2n}}{(q;q)_{2n+1}}.
\end{eqnarray}
Using identities (38) and (39) from Slater's list \cite{LS}, \eqref{disspar} and \eqref{scdiss} yield
\begin{equation}\label{ppscdiss}
\begin{split}
&\sum_{n\geq 0}p_{\text{sc}}(2n)q^n=\sum_{n\geq 0}\dfrac{q^{2n^2}}{(q;q)_{2n}}=\dfrac{(-q^3,-q^5,q^8;q^8)_\infty}{(q^2;q^2)_\infty}=\dfrac{(q;q)_\infty}{(-q;q)_\infty}\sum_{n\geq 0}p(2n)q^n\\
&\sum_{n\geq 0}p_{\text{sc}}(2n+1)q^n=\sum_{n\geq 0}\dfrac{q^{2n^2+2n}}{(q;q)_{2n+1}}=\dfrac{(-q,-q^7,q^8;q^8)_\infty}{(q^2;q^2)_\infty}=\dfrac{(q;q)_\infty}{(-q;q)_\infty}\sum_{n\geq 0}p(2n+1)q^n.
\end{split}
\end{equation}
Using \eqref{jtp1} on the extreme right-hand sides of each identity in \eqref{ppscdiss} and comparing coefficients of $q^n$ on both sides of the identities together yield \eqref{claim1}. 
\end{proof}
\subsection{Proof of Theorem \ref{ideNp}}
From Theorem \ref{spt-ana} and \eqref{corBLO},
\begin{align}\label{iden}
&\sum_{\substack{r,s,n\geq 0\\m\in\mathbb{Z}}}N(r,s,m-2s+2r,n)(-1)^{r+s+m}z^m q^n\nonumber\\
&\qquad\qquad\qquad=(1-z)(1-z^{-1})\left[\dfrac{-1}{z(1-z^{-1}+z^{-2})}\cdot\dfrac{(z^{-2}q,z^2q)_\infty}{(-z^{-1}q,-zq)_\infty}\right]+\dfrac{(1+z^{-1})}{z(1+z^{-3})}.    
\end{align}
The idea is to take the fourth derivative on both sides of the above identity with respect to $z$, let $z=1$, and then equate the coefficients of $q^n$ on both sides of the resulting identity. We first concentrate on the right-hand side.

Invoking Lemma \ref{diffk} and using the definition of $D(z,q)$ in \eqref{Dzq}, it is seen using routine simplification that
\begin{align}\label{4thderiD}
&-\dfrac{1}{24}\left[\dfrac{d^4}{dz^4}\left\{(1-z)(1-z^{-1})\left(\dfrac{-1}{z(1-z^{-1}+z^{-2})}\cdot D(z,q)\right)+\dfrac{(1+z^{-1})}{z(1+z^{-3})}\right\}\right]_{z=1}\notag\\
&=-\dfrac{1}{2}D''(1,q)\notag\\
&=(q;q)_\infty^2(q;q^2)_\infty^2\left\{3\sum_{n=1}^\infty\dfrac{nq^n}{1-q^n}+2\sum_{n=1}^\infty\dfrac{(2n-1)q^{2n-1}}{1-q^{2n-1}}\right\},
\end{align}
where in the last step we invoked Theorem \ref{Quintupledoublederi}. 
Note that differentiating Euler's generating function for $p(n)$ leads to \cite[Equation (3.3)]{gea266}
\begin{equation}\label{diffeuler}
\sum_{n=1}^{\infty}np(n)q^n=\frac{1}{(q;q)_{\infty}}\sum_{n=1}^{\infty}\frac{nq^n}{1-q^n}.    
\end{equation}
Since $p_{\text{sc}}(n)$ equals the number of partitions of $n$ into distinct odd parts, we have
\begin{equation*}
\sum_{n=1}^{\infty}p_{\text{sc}}(n)q^n=(-q;q^2)_{\infty}.   
\end{equation*}
Replacing $q$ by $-q$ in the above identity and then differentiating both sides with respect to $q$ leads us to
\begin{equation}\label{diffpsc}
\sum_{n=1}^{\infty}(-1)^n np_{\text{sc}}(n)q^n=-(q;q^2)_{\infty}\sum_{n=1}^\infty\dfrac{(2n-1)q^{2n-1}}{1-q^{2n-1}}.
\end{equation}
Therefore from \eqref{diffeuler} and \eqref{diffpsc}, we deduce that
\begin{align}\label{logdiffsp}
3\sum_{n=1}^\infty\dfrac{nq^n}{1-q^n}+2\sum_{n=1}^\infty\dfrac{(2n-1)q^{2n-1}}{1-q^{2n-1}}
&=3(q;q)_\infty\sum_{n\geq 1}np(n)q^n-\dfrac{2}{(q;q^2)_\infty}\sum_{n\geq 1}(-1)^n np_{\text{sc}}(n)q^n\notag\\
&=3(q;q)_\infty\sum_{n\geq 1}np(n)q^n-2(-q;q)_\infty\sum_{n\geq 1}(-1)^n np_{\text{sc}}(n)q^n,
\end{align}
where in the last step, we employed the elementary result $1/(q;q^2)_{\infty}=(-q;q)_{\infty}$. Thus taking the fourth derivative on both sides of \eqref{iden} with respect to $z$ and then letting $z=1$, employing \eqref{4thderiD} and \eqref{logdiffsp}, and then multiplying both sides by $(-q;q)_\infty/(q;q)_\infty^2$, we obtain
\begin{eqnarray}\label{equivrhs}
&&-\dfrac{1}{24}\dfrac{(-q;q)_\infty}{(q;q)_\infty^2}\sum_{n=1}^{\infty}q^n\sum_{\substack{r,s\geq 0\\m\in\mathbb{Z}}}m(m-1)(m-2)(m-3)N(r,s,m-2s+2r,n)(-1)^{r+s+m}\notag\\&&\hspace{2cm}=3\dfrac{(q;q)_\infty}{(-q;q)_\infty}\sum_{n\geq 1}np(n)q^n-2\sum_{n\geq 1}(-1)^nnp_{\text{sc}}(n)q^n.
\end{eqnarray}
By an application of the Jacobi triple product identity \cite[Theorem 1.3.3]{spirit},
\begin{eqnarray}\label{jtp1}
1+2\sum_{j=1}^\infty(-1)^jq^{j^2}=\sum_{j\in\mathbb{Z}} (-1)^jq^{j^2}=\dfrac{(q;q)_\infty}{(-q;q)_\infty}
\end{eqnarray}
so that the right-hand side of \eqref{equivrhs} can be rewritten as
\begin{align}\label{finrhs}
&3\dfrac{(q;q)_\infty}{(-q;q)_\infty}\sum_{n\geq 1}np(n)q^n-2\sum_{n\geq 1}(-1)^nnp_{\text{sc}}(n)q^n\notag\\
&=\sum_{n\geq 1}\left[3np(n)+6\sum_{j\geq 1 }(-1)^j(n-j^2)p(n-j^2)-2(-1)^nnp_{\text{sc}}(n)\right]q^n\notag\\
&=\sum_{n\geq 1}\left[(3-2(-1)^n)np(n)+6\sum_{j\geq 1 }(-1)^j\left\{(n-j^2)p(n-j^2)-2(-1)^{n}np(n-2j^2)\right\}\right]q^n,
\end{align}
where in the last step, we invoked Lemma \ref{claim1lemma}. Lastly, observe that replacing $z$ by $z^{-1}$ in \eqref{corBLO} results in
\begin{equation}\label{symmetry}
N(r,s,m-2s+2r,n)=N(r,s,-m-2s+2r,n),
\end{equation}
which, in turn, implies that for fixed $r, s$ and $n$,
\begin{align}\label{symmetry1}
&\sum_{m\in\mathbb{Z}}m(m-1)(m-2)(m-3)N(r,s,m-2s+2r,n)(-1)^{m}  \nonumber\\  
&=\sum_{m\in\mathbb{Z}}(-1)^{m}m^2(m^2-11)N(r,s,m-2s+2r,n).
\end{align}
Hence invoking Lemma \ref{parlemma}, \eqref{equivrhs}, \eqref{finrhs} and \eqref{symmetry1}, the result now follows by comparing the coefficients of $q^n$ on both sides and by expressing the resulting right-hand side as a bilateral series.
\qed
\subsection{Proof of Theorem \ref{Th3}}
We compute the fourth derivative of the identity in Theorem \ref{spt-ana} with respect to $z$ and then let $z=1$. Using Lemma \ref{diffk}, we obtain
\begin{eqnarray}\label{deri4th}
-\dfrac{1}{4!}\Bigg[\dfrac{d^4}{dz^4}(1-z)(1-z^{-1})f(z)\Bigg]_{z=1}&=&\sum_{\ell=2}^4\dfrac{(-1)^\ell}{(\ell-2)!}\Bigg[\dfrac{d^{\ell-2}}{dz^{\ell-2}}f(z)\Bigg]_{z=1}\notag\\
&=&f(1)-f'(1)+\dfrac{f''(1)}{2}.
\end{eqnarray}
First, let $f(z)$ be defined by
\begin{eqnarray}\label{fdef}
f(z):=(1+z)(1+z^{-1})\sum_{n=1}^\infty\dfrac{(z^2q)_{n-1}(z^{-2}q)_{n-1}q^n}{(-zq)_n(-z^{-1}q)_n}
\end{eqnarray}
so that
\begin{eqnarray}\label{LHS1}
1+(1-z)(1-z^{-1})f(z)=\sum_{n\geq 0}\dfrac{(z^{2};q)_n(z^{-2};q)_nq^n}{(-zq;q)_n(-z^{-1}q;q)_n}.
\end{eqnarray}
Clearly, we have
\begin{equation}\label{term1}
f(1)=4\sum_{n=1}^\infty\dfrac{(q)_{n-1}^2q^n}{(-q)_n^2}.    
\end{equation}
By logarithmic differentiation, it follows that 
\begin{equation}\label{term21}
f'(z)=f(z)\left(\dfrac{1}{1+z}-\dfrac{z^{-2}}{1+z^{-1}}+\dfrac{\left(\displaystyle\sum_{n=1}^\infty\dfrac{(z^2q)_{n-1}(z^{-2}q)_{n-1}q^n}{(-zq)_n(-z^{-1}q)_n}\right)'}{\displaystyle\sum_{n=1}^\infty\dfrac{(z^2q)_{n-1}(z^{-2}q)_{n-1}q^n}{(-zq)_n(-z^{-1}q)_n}}\right).    
\end{equation}
Next, we have
\begin{eqnarray}\label{term22}
\left(\displaystyle\sum_{n=1}^\infty\dfrac{(z^2q)_{n-1}(z^{-2}q)_{n-1}q^n}{(-zq)_n(-z^{-1}q)_n}\right)'&=&\sum_{n=1}^\infty\dfrac{(z^2q)_{n-1}(z^{-2}q)_{n-1}q^n}{(-zq)_n(-z^{-1}q)_n}\left\{\sum_{k=1}^{n-1}\dfrac{-2zq^k}{1-z^2q^k}+\sum_{k=1}^{n-1}\dfrac{2z^{-3}q^k}{1-z^{-2}q^k}\right.\notag\\
&&\left.-\;\sum_{k=1}^n\dfrac{q^k}{1+zq^k}+\sum_{k=1}^n\dfrac{z^{-2}q^k}{1+z^{-1}q^k}\right\}.   
\end{eqnarray}
Thus \eqref{term22} yields
\begin{eqnarray}\label{term23}
\left[\left(\displaystyle\sum_{n=1}^\infty\dfrac{(z^2q)_{n-1}(z^{-2}q)_{n-1}q^n}{(-zq)_n(-z^{-1}q)_n}\right)'\right]_{z=1}=0.
\end{eqnarray}
Hence \eqref{term21} and \eqref{term23} yield
\begin{eqnarray}\label{derif0}
f'(1)&=&0.
\end{eqnarray}
Next, using \eqref{term21} we compute the second derivative of $f(z)$ to get
\begin{eqnarray}\label{term24}
f''(z)=f'(z)\cdot S(z)+f(z)\cdot S'(z), 
\end{eqnarray}
where 
\begin{eqnarray}\label{term25}
S(z):=\dfrac{1}{1+z}-\dfrac{z^{-2}}{1+z^{-1}}+\dfrac{\left(\displaystyle\sum_{n=1}^\infty\dfrac{(z^2q)_{n-1}(z^{-2}q)_{n-1}q^n}{(-zq)_n(-z^{-1}q)_n}\right)'}{\displaystyle\sum_{n=1}^\infty\dfrac{(z^2q)_{n-1}(z^{-2}q)_{n-1}q^n}{(-zq)_n(-z^{-1}q)_n}}.
\end{eqnarray}
We compute $S'(z)$ first. Let us further put 
\begin{eqnarray}\label{term26}
S_1(z):=\left(\displaystyle\sum_{n=1}^\infty\dfrac{(z^2q)_{n-1}(z^{-2}q)_{n-1}q^n}{(-zq)_n(-z^{-1}q)_n}\right)',\hspace{1cm}S_2(z):=\displaystyle\sum_{n=1}^\infty\dfrac{(z^2q)_{n-1}(z^{-2}q)_{n-1}q^n}{(-zq)_n(-z^{-1}q)_n}.
\end{eqnarray}
Then \eqref{term25} and \eqref{term26} yield
\begin{eqnarray}\label{term27}
S'(z)=-\dfrac{1}{(1+z)^2}+\dfrac{1+2z}{(z+z^2)^2}+\dfrac{S_2(z)\cdot S_1'(z)-S_1(z)\cdot S_2'(z)}{S_2^{2}(z)}.
\end{eqnarray}
Next, we note from \eqref{fdef} and \eqref{term26} that $f(z)=(1+z)(1+z^{-1})S_{2}(z)$ and from \eqref{term23} that $S_{1}(1)=0$. Then from \eqref{derif0}, \eqref{term24} and \eqref{term27} that
\begin{eqnarray}\label{prefinal}
f''(1)=4\cdot S_2(1)\cdot \left(-\dfrac{1}{4}+\dfrac{3}{4}+\dfrac{S_1'(1)}{S_2(1)}\right)=2\sum_{n=1}^\infty\dfrac{(q)_{n-1}^2q^n}{(-q)_n^2}+4\cdot S_1'(1).
\end{eqnarray}
Thus, it remains to calculate $S_1'(1)$. Before we do that, we note that $S_1(1)=0$ and this precisely happens since the quantity inside curly braces in the right-hand side of \eqref{term22} is zero. Let us call this quantity $C(z)$. Thus, in order to calculate $S_1(1)'$, we need only calculate $C'(1)$. 
\begin{eqnarray}\label{Cderi}
C'(z)&=&\sum_{k=1}^{n-1}\dfrac{(1-z^2q^k)\cdot(-2q^k)-(-2zq^k)\cdot(-2zq^k)}{(1-z^2q^k)^2}+\sum_{k=1}^{n-1}\dfrac{-2q^k(3z^2-q^k)}{(z^3-zq^k)^2}\notag\\
&-&\sum_{k=1}^n\dfrac{-q^{2k}}{(1+zq^k)^2}+\sum_{k=1}^n\dfrac{-q^k(2z+q^k)}{(z^2+zq^k)^2}.
\end{eqnarray}
Thus \eqref{Cderi} implies
\begin{eqnarray}\label{Cderival}
C'(1)&=&-8\sum_{k=1}^{n-1}\dfrac{q^k}{(1-q^k)^2}-2\sum_{k=1}^n\dfrac{q^k}{(1+q^k)^2}\notag\\
&=&-2\sum_{k=1}^{n-1}\dfrac{q^k(5+6q^k+5q^{2k})}{(1-q^{2k})^2}-\dfrac{2q^n}{(1+q^n)^2}.
\end{eqnarray}
Hence \eqref{term22},\eqref{term26}, \eqref{prefinal} and \eqref{Cderival} yield
\begin{eqnarray}\label{derif2}
f''(1)=2\sum_{n=1}^\infty\dfrac{(q)_{n-1}^2q^n}{(-q)_n^2}-8\sum_{n=1}^\infty\dfrac{(q)_{n-1}^2q^n}{(-q)_n^2}\left(\sum_{k=1}^{n-1}\dfrac{q^k(5+6q^k+5q^{2k})}{(1-q^{2k})^2}+\dfrac{q^n}{(1+q^n)^2}\right).
\end{eqnarray}
Thus \eqref{deri4th}, \eqref{LHS1}, \eqref{term1}, \eqref{derif0} and \eqref{derif2} yield
\begin{eqnarray}\label{LHS4thderi}
&&-\dfrac{1}{24}\left[\dfrac{d^4}{dz^4}\sum_{n\geq 0}\dfrac{(z^{2};q)_n(z^{-2};q)_nq^n}{(-zq;q)_n(-z^{-1}q;q)_n}\right]_{z=1}=5\sum_{n=1}^\infty\dfrac{(q)_{n-1}^2q^n}{(-q)_n^2}\notag\\&&\hspace{3cm}-4\sum_{n=1}^\infty\dfrac{(q)_{n-1}^2q^n}{(-q)_n^2}\left(\sum_{k=1}^{n-1}\dfrac{q^k(5+6q^k+5q^{2k})}{(1-q^{2k})^2}+\dfrac{q^n}{(1+q^n)^2}\right).    
\end{eqnarray}
Along with \eqref{spt-anaeqn} and \eqref{4thderiD}, this implies the result.
\qed
\subsection{Proof of Theorem \ref{spt-mod}}
Using Theorem \ref{Th3} and \eqref{blospt}, we have
\begin{eqnarray}\label{Ded13}
&4\displaystyle\sum_{n=1}^\infty\dfrac{(q)_{n-1}^2q^n}{(-q)_n^2}\left(\sum_{k=1}^{n-1}\dfrac{q^k(5+6q^k+5q^{2k})}{(1-q^{2k})^2}+\dfrac{q^n}{(1+q^n)^2}\right)
=-\dfrac{5}{4}\left(\dfrac{(q)_\infty^2}{(-q)_\infty^2}-1\right)&\notag\\&-\;\dfrac{(q)_\infty^2}{(-q)_\infty^2}\left\{5\displaystyle\sum_{n=1}^\infty\dfrac{nq^n}{1-q^n}-4\sum_{n=1}^\infty\dfrac{nq^{2n}}{1-q^{2n}}\right\}.&
\end{eqnarray}
Note that $\dfrac{(q)_\infty^2}{(-q)_\infty^2}=\dfrac{(q)_\infty^4}{(q^2;q^2)_\infty^2}$ and by logarithmic differentiation we have
\begin{equation}\label{logdiffeta}
\dfrac{d}{dq}\left(\dfrac{(q)_\infty^5}{(q^2;q^2)_\infty^2}\right)=\dfrac{(q)_\infty^5}{(q^2;q^2)_\infty^{2}}\left\{-5\displaystyle\sum_{n=1}^\infty\dfrac{nq^{n-1}}{1-q^n}+4\sum_{n=1}^\infty\dfrac{nq^{2n-1}}{1-q^{2n}}\right\}.    
\end{equation}
Thus, \eqref{logdiffeta} gives
\begin{align}\label{derieta13}
\dfrac{(q)_\infty^2}{(-q)_\infty^2}\left\{5\displaystyle\sum_{n=1}^\infty\dfrac{nq^n}{1-q^n}-4\sum_{n=1}^\infty\dfrac{nq^{2n}}{1-q^{2n}}\right\}=-\dfrac{q^{25/24}}{\eta(\tau)}\cdot\dfrac{d}{dq}\left(\dfrac{(q)_\infty^5}{(q^2;q^2)_\infty^2}\right)=-\dfrac{q^{25/24}}{\eta(\tau)}\cdot\dfrac{d}{dq}\left(q^{-1/24}\dfrac{\eta(\tau)^5}{\eta(2\tau)^2}\right)
\end{align}
and from \cite[Theorem 1.1]{Robert}, it follows that 
\begin{equation}\label{etatheta}
\dfrac{\eta(\tau)^5}{\eta(2\tau)^2}=\sum_{n=1}^\infty\left(\dfrac{n}{12}\right)nq^{\frac{n^2}{24}}.
\end{equation}
Thus
\begin{eqnarray}\label{derieta14}
\dfrac{d}{dq}\left(q^{-1/24}\dfrac{\eta(\tau)^5}{\eta(2\tau)^2}\right)=\dfrac{1}{24}\sum_{n=1}^\infty\left(\dfrac{n}{12}\right)n(n^2-1)q^{\frac{n^2-1}{24}-1}.
\end{eqnarray}
Combining \eqref{Ded13}, \eqref{derieta13} and \eqref{derieta14}, we get
\begin{align}
&4\displaystyle\sum_{n=1}^\infty\dfrac{(q)_{n-1}^2q^n}{(-q)_n^2}\left(\sum_{k=1}^{n-1}\dfrac{q^k(5+6q^k+5q^{2k})}{(1-q^{2k})^2}+\dfrac{q^n}{(1+q^n)^2}\right)\notag\\
&\hspace{6cm}=-\dfrac{5}{4}\left(\dfrac{(q)_\infty^2}{(-q)_\infty^2}-1\right)+\dfrac{1}{24\eta(\tau)}\sum_{n=1}^\infty\left(\dfrac{n}{12}\right)n(n^2-1)q^{\frac{n^2}{24}}\notag\\
&\hspace{6cm}=-\dfrac{5}{4}\left(\dfrac{(q)_\infty^2}{(-q)_\infty^2}-1\right)-\dfrac{1}{24}\dfrac{\eta(\tau)^4}{\eta(2\tau)^2}+\dfrac{1}{24}\cdot\dfrac{\theta(\tau)}{\eta(\tau)}
\end{align}
where we have used \eqref{etatheta} in the last step and where $\theta(\tau)$ is defined as in the theorem. The result now follows. It remains to show that $\theta(2\tau)$ satisfies the transformation in \eqref{thetatrans}. This follows by choosing $N=6,\;h=1,\;P(m)=m^3,\;A=[6]$ in \cite[Proposition 2.1]{Shimura}. 
\qed

\begin{center}
\textbf{Acknowledgements}
\end{center}
The authors sincerely thank George Andrews and Jeremy Lovejoy for their helpful comments. The first author is supported by the Swarnajayanti Fellowship Grant SB/SJF/2021-22/08 of SERB (Govt. of India) and thanks the agency for the support. The second author is an institute postdoctoral fellow at IIT Gandhinagar under the project IP/IITGN/MATH/AD/2122/15. He thanks the institute for the support.

\end{document}